\documentclass[reqno,11pt,twoside]{amsart}

\usepackage{xypic}
\usepackage{graphics,amssymb,comment}
\usepackage{latexsym}
\usepackage{mathrsfs}
\usepackage{chngcntr,color}
\usepackage{hyperref}
\hypersetup{ 
colorlinks,
citecolor=black,
filecolor=green,
linkcolor=black,
urlcolor=green
} 
\xyoption{curve}

\newtheorem{lemma}{Lemma}[section]
\newtheorem{proposition}[lemma]{Proposition}
\newtheorem{theorem}[lemma]{Theorem}
\newtheorem{corollary}[lemma]{Corollary}

\theoremstyle{definition}

\newtheorem{definition}[lemma]{Definition}

\newtheorem{question}[lemma]{Question}

\newcommand{\RHom}{\mathrm{RHom}}
\newcommand{\DPic}{\mathrm{DPic}}
\newcommand{\End}{\mathrm{End}}
\newcommand{\Pic}{\mathrm{Pic}}

\newcommand{\Aut}{\mathrm{Aut}}
\newcommand{\Spec}{\mathrm{Spec}}
\newcommand{\Out}{\mathrm{Out}}
\newcommand{\ox}{\otimes}
\renewcommand{\L}{\mathrm{L}}

\title[Derived Picard group of an affine Azumaya algebra]{The derived Picard group of an affine Azumaya algebra}
\date{\today}
\author{Cris Negron}
\address{Department of Mathematics\\Louisiana State University\\
Baton Rouge, LA 70803, USA}
\email{cnegron@lsu.edu}
\thanks{This work was supported by NSF Postdoctoral Research Fellowship DMS-1503147}

\begin{document}
\maketitle

\begin{abstract}
We describe the derived Picard group of an Azumaya algebra $A$ on an affine scheme $X$ in terms of global sections of the constant sheaf of integers on $X$, the Picard group of $X$, and the stabilizer of the Brauer class of $A$ under the action of $\Aut(X)$.  In particular, we find that the derived Picard group of an Azumaya algebra is generally not isomorphic to that of the underlying scheme.  In the case of the trivial Azumaya algebra, our result refines Yekutieli's description of the derived Picard group of a commutative algebra.  We also get, as a corollary, an alternate proof of a result of Antieau which relates derived equivalences to Brauer equivalences for affine Azumaya algebras.  The example of a Weyl algebra in finite characteristic is examined in some detail.
\end{abstract}

\section{Introduction}

We work over a commutative base ring $k$ over which all algebras are assumed to be flat.
\par

In this paper we provide a description of the derived Picard group of an affine Azumaya algebra.  The derived Picard group $\DPic(A)$ of an algebra $A$ is the group of isoclasses of {\it tilting complexes} in $D^b(A\ox_k A^{op})$ (see~\ref{sect:dpg}).  The derived Picard group is known to be an invariant of the derived category of $A$ when $k$ is a field and, as is explained below, it is strongly related to the Hochschild cohomology of $A$.  An Azumaya algebra can be seen as a type of globalization of a central simple algebra (see~\ref{sect:Az}).  A Weyl algebra in finite characteristic, for example, is Azumaya over its center~\cite{revoy73}.

\begin{theorem}[\ref{thm:dpicaz}]\label{thm:} 
Let $R$ be a commutative algebra, $X=\Spec(R)$, and $A$ be an Azumaya algebra over $R$.  Then there is a group isomorphism
\[
\DPic(A)\cong \left(\Gamma(X,\underline{\mathbb{Z}})\times \Pic(X)\right)\rtimes_{\alpha}\Aut(X)_{[A]}.
\]
\end{theorem}

Here $\underline{\mathbb{Z}}$ is the constant sheaf of the integers on $X$, and $\Aut(X)_{[A]}$ is the stabilizer of the Brauer class of $A$ under the action of $\Aut(X)=\Aut_{k\text{-}\mathrm{schemes}}(X)$ on the Brauer group $\mathrm{Br}(X)$ by pushforward.  The element $\alpha$ is a $\Pic(X)$-valued $2$-cocycle
\[
\alpha:\Aut(X)_{[A]}\times \Aut(X)_{[A]}\to \Pic(X).
\]
The particular form of Theorem~\ref{thm:} for central simple algebras is given in Section~\ref{sect:CSA}.  In the case in which $A=R$ Theorem~\ref{thm:} refines a well established result of Yekutieli and Rouquier-Zimmerman~\cite{yekutieli99,rouquierzimmerman03,yekutieli15} (see also~\cite{yekutieli10}).

\begin{corollary}[\ref{cor:dpiccom}]
For any commutative algebra $R$, and $X=\Spec(R)$, there is a group isomorphism
\[
\DPic(R)\cong \left(\Gamma(X,\underline{\mathbb{Z}})\times\Pic(X)\right)\rtimes\Aut(X).
\]
\end{corollary}

We recall again that $\Aut(X)=\Aut_{k\text{-schemes}}(X)$.  In~\cite{yekutieli99,yekutieli15} the above result is given for algebras $R$ with a finite decomposition $R=\prod_{i=1}^n R_i$ into indecomposable algebras $R_i$ (e.g. take $R$ to be Noetherian).  As a corollary to Proposition~\ref{prop:deq_as_meq} below, we also get an alternate proof of a result of Antieau which relates derived equivalences to Brauer equivalences (see Corollary~\ref{cor:antieau}).
\par

We consider, as an example, the Weyl algebra $A_n(k)$ in finite characteristic, which is Azumaya over a polynomial ring $Z_n(k)$ in $2n$ variables~\cite{revoy73}.  We have $\mathbb{A}^{2n}_k=\Spec(Z_n(k))$.

\begin{theorem}[\ref{thm:az}--\ref{cor:az}]\label{thm::}
Suppose $k$ is a field of finite characteristic other than $2$.  For the Weyl algebra $A_n(k)$ we have the following information:
\begin{enumerate}
\item $\DPic(A_n(k))=\mathbb{Z}\times \Aut(\mathbb{A}^{2n}_k)_{[A_n(k)]}$.
\item There is a canonical group embedding $\DPic(A_n(k))\to \DPic(Z_n(k))$.
\item The stabilizer of the Brauer class of $A_n(k)$ is a proper subgroup in $\Aut(\mathbb{A}^{2n}_k)$ which is of infinite index when $k$ is an infinite field.
\item The embedding of {\rm (2)} is \emph{not} an isomorphism.
\end{enumerate}
\end{theorem}

We expect that the index of the stabilizer is still infinite even when $k$ is finite.  In order to prove Theorem~\ref{thm::} we provide an analysis of the behavior of the Brauer class of $A_n(k)$ which may be of independent interest.
\par

The present study is motivated by some recent results on derived invariants for Azumaya algebras, most notably the work of Tabuada and Van den Bergh~\cite{tabuadavandenbergh15}.  In~\cite{tabuadavandenbergh15} the authors show that, under a certain restriction on the characteristic, for any {\it additive invariant} $E$, and Azumaya algebra $A$ with center $R$, there will be a canonical isomorphism $E(A)\cong E(R)$.\footnote{The results of~\cite{tabuadavandenbergh15} hold in the much broader setting of sheaves of Azumaya algebras $\mathscr{A}$ on quasi-compact quasi-separated schemes $X$ with a finite number of connected components.  The characteristic restriction is as follows: for a $K$-linear additive invariant $E$, with $K$ a commutative ring, the product $r$ of the ranks of $\mathscr{A}$ on the components of $X$ must be a unit in $K$.}  The work of~\cite{tabuadavandenbergh15} comes after related results of Corti\~{n}as and Weibel on Hochschild and cyclic homology, and Hazrat, Hoobler, and Millar on algebraic $K$-theory~\cite{cortinasweibel94,hazratmillar10,hazrathoobler13}.  
\par

An emerging principle seems to be that derived invariants for $A$ and $R$ will be very strongly related, and often isomorphic.  The description of the derived Picard group in Theorem~\ref{thm:} serves to illuminate the boundaries of this principle ({\it cf.} the description of cyclic homology in~\cite{cortinasweibel94} and $K$-theory in~\cite{hazratmillar10,hazrathoobler13}).
\par

We are also motivated by the relationship between the derived Picard group and Hochschild cohomology established by Keller in~\cite{keller04}.  When $k$ is a field, what we have described above is the group of $k$-points for a certain group valued functor
\[
\DPic^A:\mathrm{comm.\ dg\ algebras}\to \mathrm{groups}.
\]
The Hochschild cohomology is the graded tangent space of this group valued functor at the identity, i.e. the collective kernels of the maps $\DPic^A(k[\epsilon]/(\epsilon^2))\to\DPic^A(k)$ where the degree of $\epsilon$ varies (see \cite[\S 4]{keller04}).  Although it has been suggested (in personal communications) that the Hochschild cohomologies for $A$ and $R$ should agree, we are unsure of what to expect from this invariant at the moment.  Our description of $\DPic(A)=\DPic^A(k)$ above tells us that the ambient groups will in fact differ in general.

\subsection{Organization of the paper}

Section~\ref{sect:bg} is dedicated to background material.  In Section~\ref{sect:deq} we analyze tilting complexes between Azumaya algebras and in Section~\ref{sect:brauer} we use our findings to relate derived equivalences to $\Aut(X)$-orbits in the Brauer group.  In Section~\ref{sect:dpicA} we prove Theorem~\ref{thm:}.  Section~\ref{sect:Weyl} is dedicated to an analysis of the Weyl algebra in finite characteristic.


\section{Background}
\label{sect:bg}

\subsection{Notations}

Throughout $k$ will be a commutative base ring.  The shift $\Sigma M$ of a complex $M$ is the complex with $(\Sigma M)^i=M^{i+1}$ and negated differential.  By an algebra we mean a $k$-algebra, and $\ox=\ox_k$.  By an unadorned $\Aut$ we mean either $\Aut_{k\text{-}\mathrm{alg}}$ or $\Aut_{k\text{-}\mathrm{schemes}}$.
\par

Given an algebra $A$ we let $D^b(A)$ denote the bounded derived category of $A$.  A perfect complex is an object in $D^b(A)$ which is isomorphic to a bounded complex of finitely generated projective $A$-modules.

\subsection{Bimodules}

Let $A$ and $B$ be algebras.  A $B\ox A^{op}$-module $M$ is said to be an {\it invertible bimodule} if there is a $A\ox B^{op}$-module $M^\vee$ so that $M\ox_A M^\vee\cong B$ and $M^\vee\ox_B M\cong A$ in the categories of $B$-bimodules and $A$-bimodules respectively.
\par

For a commutative algebra $R$, a $R$-bimodule $M$ is said to be a {\it symmetric bimodule} if for each $m\in M$ and $r\in R$ we have $mr=rm$.  If $A$ and $B$ are $R$-algebras and $M$ is a $B\ox A^{op}$-module, we say $M$ is {\it $R$-central} if its restriction to $R\ox R^{op}$ is symmetric.   This is the same thing as being a $B\ox_R A^{op}$-module.  All bimodules are assumed to be $k$-symmetric.

\subsection{The derived Picard group}
\label{sect:dpg}

The derived Picard group $\DPic(A)$ of an algebra $A$ was introduced independently by Yekutieli and Rouquier-Zimmermann~\cite{yekutieli99,rouquierzimmerman03}.  This group can be defined succinctly as the collection of isoclassses of invertible objects in the monoidal category $(D^b(A\ox A^{op}),\ox^{\L}_A)$.  Via Rickard's study of derived equivalences~\cite{rickard91}, one finds that the derived Picard group is a derived invariant in the sense that any $k$-linear triangulated equivalence $D^b(A)\overset{\sim}\to D^b(B)$ implies an isomorphism of groups $\DPic(A)\cong\DPic(B)$, provided both $A$ and $B$ are projective over $k$.
\par

We let $\DPic(A,B)$ denote the collection of isoclasses of objects $T$ in $D^b(B\ox A^{op})$ which admits a object $T^\vee$ in $D^b(A\ox B^{op})$ such that $T\ox_A^\L T^\vee\cong B$ and $T^\vee\ox_B^\L T\cong A$ in their respective derived categories of bimodules.  Such an object $T$ will be called a {\it tilting complex} and $T^\vee$ will be called its dual, or quasi-inverse.\footnote{These will be tilting complexes in the usual sense of Rickard~\cite[Theorem 1.6]{yekutieli99}, \cite{rickard91}.}
\par

Obviously $\DPic(A)=\DPic(A,A)$.  We denote the class of a tilting complex $T$ in $\DPic(A,B)$ by $[T]$, although we will often abuse language and refer to $T$ as an object in $\DPic(A,B)$.
\par

Note that for any object $T$ in $\DPic(A,B)$ we will get a triangulated equivalence $T\ox_A^\L-:D^b(A)\overset{\sim}\to D^b(B)$ with quasi-inverse $T^\vee\ox^\L_B-$.  From this one can conclude that any tilting complex is a perfect complex over $A$ and $B$ independently, and that the two dg algebra maps
\[
A^{op}\to\RHom_B(T,T)\ \ \mathrm{and}\ \ B\to \RHom_A(T,T) 
\]
given by the right and left actions will be quasi-isomorphisms.

\subsection{Azumaya algebras}
\label{sect:Az}

\begin{definition}
An Azumaya algebra over a commutative ring $R$ is an $R$-algebra $A$ which is finitely generated and projective over $R$ and is such that the action map $A\ox_R A^{op}\to \End_R(A,A)$ is an isomorphism of $R$-algebras.
\end{definition}

It is a fact that any Azumaya algebra $A$ over $R$ has $\mathrm{Center}(A)=R$~\cite[Proposition 1.1]{milne80}.  There are also many other equivalent Azumaya conditions, which can be found in~\cite[Proposition 2.1]{milne80}.  Two fundamental classes of Azumaya algebras are central simple algebras over fields, and crystalline differential operators on smooth schemes~\cite{bmr08}, for example.  Some Azumaya algebras related to quantum groups can be found in~\cite{browngoodearl97} and a number of other examples of Azumaya algebras and references can be found in~\cite[\S 3]{tabuadavandenbergh15}.

\section{Derived equivalences as Morita equivalences and other fundamentals}
\label{sect:deq}

Throughout $R$ will be a commutative algebra which is flat over $k$.  All other algebras are assumed to be flat over $k$ as well.
\par

We impose this restriction so that we may derive tensor products, say $\ox_A$ for example, on the categories $B\ox A^{op}$-mod and $A\ox B^{op}$-mod via bimodules which are flat over either $A$ or $B$.  Under our flatness assumption, projective bimodule will have the desired property.

\subsection{Shifts by sections of the constant sheaf}
\label{sect:shifts}

Let $X$ be a scheme.  The constant sheaf $\underline{G}$ associated to an abelian group $G$ is the sheaf of continuous functions from $X$ to the discrete space $G$.  By continuity any section $g\in \Gamma(X,\underline{G})$ disconnects $X$ into a number of components.  Namely, one component for each element in the image of $g:X\to G$.
\par

Take $X=\Spec(R)$ and $G=\mathbb{Z}$.  Let $n$ be a global section of $\underline{\mathbb{Z}}$ and take
\[
\Lambda=\{n^{-1}(i):i\in \mathbb{Z}\}
\]
to be the set of components specified by $n$.  Note that by quasi-compactness of $X$ the cardinality of $\Lambda$ is finite.  We have $X=\coprod_{\lambda\in \Lambda}\lambda$ and we get a collection of orthogonal idempotents $\{e_\lambda\}_{\lambda}$ with $\sum_\lambda e_\lambda=1$ which specify a dual decomposition
\[
R=\prod_{\lambda\in \Lambda} R_\lambda.
\]
The collection $\{e_\lambda\}_\lambda$ can be found abstractly as a collection of orthogonal idempotents simultaneously satisfying $V(e_\lambda)=\coprod_{\{\mu\in\Lambda:\mu\neq\lambda\}} \mu$ and $1=\sum_\lambda e_\lambda$~\cite[Ex.2.25]{eisenbud}.  As usual, $V(e_\lambda)$ denotes the closed subset in $X$ defined by $e_\lambda$.

\begin{lemma}
Take a decomposition $\Spec(R)=\coprod_{\lambda\in\Lambda} \lambda$ as above, and consider two choices of corresponding orthogonal idempotents $\{e_\lambda\}_\lambda$ and $\{e'_\lambda\}_\lambda$.  Then for each $\lambda\in\Lambda$ we have $e_\lambda=r e'_\lambda$ and $e'_\lambda=r'e_\lambda$ for some $r,r'\in R$.  In particular, for any $R$-module $M$ we have $e_\lambda M=e'_\lambda M$.
\end{lemma}

\begin{proof}
It suffices to consider a decomposition $\Spec(R)=X_1\coprod X_2$.  Take $I_1=I(X_2)$, $I_2=I(X_1)$.  Let $\{e_1,e_2\}$ and $\{e'_1,e'_2\}$ be collections of orthogonal idempotents with
\[
e_i,e_i'\in I_i,\ \mathrm{and}\ e_1+e_2=e_1'+e_2'=1.
\]
Then we have $V(e_i)=V(e'_i)=V(I_i)$ which implies that the kernel of each map 
\[
R/(e_i)\to R/I_i\ \ \mathrm{and}\ \ R/(e_i')\to R/I_i
\]
consists entirely of nilpotent elements~\cite[Lemma 1.6]{liu}.  In particular, some positive power of $e_i\in I_i$ is equal to $re'_i$ for some $r\in R$, and hence $e_i=e_i^n=re'_i$.  Similarly, $e'_i=r'e_i$ for some $r'\in R$.
\end{proof}

Let $M$ be a complex of $R$-modules and $\{e_\lambda\}_\lambda$ be a finite collection of orthogonal idempotents with $1=\sum_\lambda e_\lambda$.  Then by $R$-linearity of the differential on $M$ we see that $M=\oplus_\lambda e_\lambda M$ is a chain complex decomposition.  By the above lemma the following operation is well defined.

\begin{definition}
Take $X=\Spec(R)$ and $n\in\Gamma(X,\underline{\mathbb{Z}})$.  Take also $\Lambda=\{n^{-1}(i):i\in\mathbb{Z}\}$ and let $M$ be a complex of $R$-modules.  We define the shift $\Sigma^n M$ by
\[
\Sigma^n M=\bigoplus_{\lambda\in \Lambda} (\Sigma^{n(\lambda)}e_\lambda M),
\]
where $\{e_\lambda\}_\lambda$ is any collection of orthogonal idempotents as above.
\end{definition}

One can check that for section $m$ and $n$ of the constant sheaf, with corresponding collections of idempotents $\{f_\mu\}_\mu$ and $\{e_\lambda\}_\lambda$, we have
\[
\Sigma^m(\Sigma^n M)=\bigoplus_{\mu,\nu}\left(\Sigma^{m(\mu)+n(\lambda)}(f_\mu e_\lambda) M\right)=\Sigma^{m+n}M.
\]

\subsection{Derived equivalences as Morita equivalences}
\label{sect:deq_as_meq}

We would like to prove the following proposition, which provides the foundation of our study.

\begin{proposition}\label{prop:deq_as_meq}
Let $R$ be a commutative algebra and take $X=\Spec(R)$.  Let $T$ be an object in $\DPic(A,B)$ with $A$ Azumaya over $R$.  Then $T\cong \Sigma^n M$ for a unique section $n\in \Gamma(X,\underline{\mathbb{Z}})$ and unique invertible $B\ox A^{op}$-module $M$, up to isomorphism.
\end{proposition}

Let us record some additional information before giving the proof of Proposition~\ref{prop:deq_as_meq}.
\par

We give $B$ an $R$-algebra structure as follows: We have $R=\mathrm{Center}(A)$ and get an induced isomorphism $R\cong \mathrm{Center}(B)$ via the sequence of isomorphisms
\[
\begin{array}{l}
R\overset{\mathrm{r.act}}\longrightarrow\End_{D^b(A\ox A^{op})}(A)\overset{T\ox^\L_A-}\longrightarrow \End_{D^b(B\ox A^{op})}(T)\\
\hspace{1cm}\overset{-\ox^\L_AT^\vee}\longrightarrow \End_{D^b(B\ox B^{op})}(B)\overset{\mathrm{l.act}^{-1}}\longrightarrow\mathrm{Center}(B),
\end{array}
\]
as in~\cite{yekutieli99} ({\it cf.}~\cite[Proposition 9.1]{rickard89}).  Note that although $T$ needn't be symmetric as an $R$-module under this identification $R=\mathrm{Center}(B)$, its cohomology will be.
\par

In the proof of Proposition~\ref{prop:deq_as_meq} we abuse notation and take, for any $p\in X=\Spec(R)$, $T_p$ to be the two sided localization $R_p\ox_R T\ox_R R_p$.  By~\cite[Lemma 2.6]{yekutieli99} the localization $T_p$ will be an object in $\DPic(A_p,B_p)$.

\subsection{The proof of Proposition~\ref{prop:deq_as_meq}}

Our proof is similar to that of~\cite[Theorem 1.9]{yekutieli10}.  We break the proof down into a sequence of lemmas for the sake of clarity.

\begin{lemma}\label{lem:0}
For $T$ as in Proposition~\ref{prop:deq_as_meq} and $p\in X$, the localization $T_p$ is isomorphic, in $D^b(B_p\ox A_p^{op})$, to an invertible $B_p\ox A_p$-module shifted to a single degree.  Furthermore, the localization of the cohomology $H^\bullet(T)_p$ is concentrated in a single degree.
\end{lemma}

\begin{proof}
Since localization is exact and the cohomology of $T$ is $R$-central we have $H^\bullet(T_p)=H^\bullet(T)_p$.  So it suffices to prove the result for $T_p$.
\par

Since $A_p$ is Azumaya over a local ring, it is a noncommutative local ring in the sense that it has a unique maximal two sided ideal~\cite[Proposition 1.1]{milne80}.  Thus~\cite[Theorem 2.3]{yekutieli99} implies $T_p$ is isomorphic to an invertible bimodule shifted to a single degree, and we are done.
\end{proof}

By Lemma~\ref{lem:0} we can now produce a well defined set map $n:X\to \mathbb{Z}$ taking $p$ to the unique integer $n(p)$ with $H^{-n(p)}(T)_p\neq 0$.  By the following lemma this map will be continuous, and hence a section of the constant sheaf of integers on $X$.

\begin{lemma}\label{lem:00}
For $T$ as in Proposition~\ref{prop:deq_as_meq}, and any integer $i$, the support of the cohomology $H^i(T)$ is both closed and open.  Furthermore, for $i\neq j$ we have $\mathrm{Supp}(H^i(T))\cap\mathrm{Supp}(H^j(T))= \emptyset$.
\end{lemma}

\begin{proof}
We have already seen that for each $p\in X$ there is a unique integer $n(p)$ so that $H^{-n(p)}(T)_p\neq 0$.  So disjointness of the supports is clear.  We will show that the cohomology is a locally free of finite rank over $R$ (relative to the Zariski topology of $\Spec(R)$).  Equivalently, we show that the cohomology is finitely generated and projective over $R$.
\par

Since the localization $T_p$ is isomorphic to an invertible bimodule shifted to a single degree, we may suppose $T_p$ itself is concentrated in a single degree.  For each $i$ we then have that $H^i(T)_p=H^i(T_p)$ is either $0$ or equal to $T_p$.  So each $H^i(T)_p=H^i(T_p)$ is finitely generated and projective over $A_p$.  Since $A_p$ is finitely generated and projective over $R_p$, we find that $H^i(T)_p$ has this property as well.  If we can show that each $H^i(T)$ is finitely presented over $R$ then we will have what we want, i.e. that the cohomology is locally free of finite rank over $R$ in each degree.
\par

Let $N$ be maximal with $H^N(T)\neq 0$.  Since $T$ is perfect $H^N(T)$ is finitely presented, and the fact that $H^N(T)_p$ is free of finite rank over $R_p$ at each point $p\in X$ implies that $H^N(T)$ is locally free of finite rank over $R$.  Since the support of such a module is always closed and open there is a idempotent $e\in R$ such that $\mathrm{Supp}(H^N(T))=V(e)$.
\par

From $e$ we can produce a tilting complex $T'=R[e^{-1}]\ox_R T\ox_R R[e^{-1}]$ between $A[e^{-1}]$ and $B[e^{-1}]$~\cite[Lemma 2.6]{yekutieli99}.  This new tilting complex will have cohomology concentrated in degree $<N$, and for each $i<N$ we will have
\[
H^i(T')=H^i(T)[e^{-1}]=H^i(T).
\]
We can therefore proceed by induction to find that each $H^i(T)$ is locally free of finite rank over $R$, and has support which is both closed and open.
\end{proof}

Along with the section $n\in\Gamma(X,\underline{\mathbb{Z}})$ defined above we will need to show that the cohomology $M=H^\bullet(T)$ is an invertible $B\ox A^{op}$-module, after forgetting the degree.  As an intermediate step we have

\begin{lemma}\label{lem:000}
Let $T$ be as in Proposition~\ref{prop:deq_as_meq}.  For any integer $i$ the cohomology group $H^i(T)$ is flat over both $A$ and $B$.
\end{lemma}

\begin{proof}
Take $p\in X=\Spec(R)$ and let $n:X\to \mathbb{Z}$ be as above.  As before, we replace the localization $T_p$ with an invertible bimodule shifted to a single degree.  Since $H^{-n(p)}(T)_p=H^{-n(p)}(T_p)=T_p$ is invertible it is finitely generated and projective over $A_p$, and hence flat as well.  Note that for all $i\neq -n(p)$ we have $H^i(T)_p=0$.  So these modules are projective and flat over $A_p$ as well.  Since flatness over $A$ can be checked locally on $X$ it follows that $H^{i}(T)$ is flat over $A$.  The same argument shows that the cohomology is flat over $B$.
\end{proof}

We can now give the

\begin{proof}[Proof of Proposition~\ref{prop:deq_as_meq}]
We already saw above that the map $n:X\to \mathbb{Z}$, sending $p$ to the unique integer $n(p)$ with $H^{-n(p)}(T)_p\neq 0$, is a section of the constant sheaf of integers.  We will use below the easy fact that for any ring $C$ and $C$-complex $U$ with cohomology concentrated in a single degree $i$, there is a canonical isomorphism $U\to H^i(U)$ in $D^b(C)$~\cite[Ch I.7]{hartshorne66}.
\par

Take $M=H^\bullet(T)$, considered simply as a bimodule concentrated in degree $0$.  So the cohomology of $T$ along with its implicit degree will be given by the shift $\Sigma^nM$.  Note that on each of the opens $\mathrm{Supp}(H^i(T))=n^{-1}(-i)$ the restriction of the complex $T$ will be isomorphic to its cohomology $\Sigma^n M$.  Since these opens are disjoint we will have $T\cong \Sigma^{n}M$ in $D^b(B\ox A^{op})$.  The same is true of the dual $T^\vee$ and its cohomology $M^\vee=H^\bullet(T^\vee)$.
\par

One can check that the section $n$ is such that $\Sigma^{-n}M^\vee\cong T^\vee$.  Hence
\[
A\cong T^\vee\ox_B^{\L} T\cong \Sigma^{-n+n}(T^\vee\ox_B^{\L} T)=(\Sigma^{-n}T^\vee)\ox_B^{\L}(\Sigma^nT)\cong M^\vee\ox_B^{\L} M,
\]
and similarly $B\cong M\ox_A^{\L} M^\vee$, in their respective derived categories of bimodules.  Note that, since $M$ and $M^\vee$ are flat over $A$ and $B$, by Lemma~\ref{lem:000}, we may take $M^\vee\ox^\L_BM=M^\vee\ox_B M$ and $M\ox^\L_A M^\vee=M\ox_A M^\vee$.
\par

Recall that for any ring $C$ the embedding $C\text{-mod}\to D^b(C)$ sending a module to the corresponding complex concentrated in degree $0$ is fully faithful.  Hence the above isomorphisms are isomorphisms of bimodules.  That is to say, $M$ is an invertible bimodule, as desired.
\end{proof}

\subsection{Automorphisms from equivalences}
\label{sect:autsfromeqs}

Let $A$ and $B$ be $R$-central algebras.  Take $X=\Spec(R)$.  For any invertible $B\ox A^{op}$-module $M$ we will get a unique automorphism of the center fitting into a diagram
\begin{equation}\label{eq:1}
\xymatrix{
R\ar[rr]^{f_M}\ar[rd]_{\mathrm{r.act}} & & R\ar[dl]^{\mathrm{l.act}}\\
 & \End_{B\ox A^{op}}(M),
}
\end{equation}
since the right and left actions can alternatively be identified with the applications of the functors $M\ox_A-$ and $-\ox_B M$ to $\End_{A\ox A^{op}}(A)$ and $\End_{B\ox B^{op}}(B)$ respectively.  Similarly, if we shift $M$ by some section of the constant sheaf, and take $T=\Sigma^n M$, we still get a diagram
\begin{equation}\label{eq:2}
\xymatrix{
R\ar[rr]^{f_T}\ar[rd]_{\mathrm{r.act}} & & R\ar[dl]^{\mathrm{l.act}}\\
 & \End_{D^b(B\ox A^{op})}(T)
}
\end{equation}
for a unique automorphism $f_T$.  (Note that $f_T=f_M$ here.)

\begin{definition}
For $R$-central algebras $A$ and $B$, invertible $B\ox A^{op}$-module $M$, section $n\in \Gamma(\Spec(R),\underline{\mathbb{Z}})$, and $T=\Sigma^{n}M$, we take $\phi^M=\Spec(f_M^{-1})$ and $\phi^T=\Spec(f_T^{-1})$ where the $f_?$ are as in (\ref{eq:1}) and (\ref{eq:2}).
\end{definition}

Given an automorphism $\phi\in \Aut(X)$ with dual algebra automorphism $\bar{\phi}\in\Aut(R)$, and $R$-algebra $A$, we let $\phi_\ast A$ denote the ring $A$ with new $R$-algebra structure $R\overset{\bar{\phi}}\to R\to A$.  (This is just the usual pushforward.)  From the above information we get

\begin{lemma}\label{lem:auts}
Let $A$ and $B$ be $R$-central algebras, $M$ be an invertible $B\ox A^{op}$-module, and $T$ be some shift $\Sigma^n M$ by a section of the constant sheaf on $\Spec(R)$.  Then $M$ is an invertible $B\ox_R (\phi^M_\ast A^{op})$-module, and $T$ is an invertible object in $D^b\left(B\ox_R (\phi^T_\ast A^{op}))\right)$.
\end{lemma}

\section{Some information on the Brauer group}
\label{sect:brauer}

\subsection{The Brauer group via bimodules}
\label{sect:braueralt}

For $X=\Spec(R)$ the {\it Brauer group} is defined as the set
$$
\mathrm{Br}(X):=\{\text{Azumaya algebras on }R\}/{\sim},
$$
where $\sim$ is the equivalence relation proposing $A\sim B$ if and only if there exist finitely generated projective $R$-modules $V$ and $W$ and an $R$-algebra isomorphism
$$
A\ox_R\End_R(V)\cong B\ox_R\End_R(W).
$$
We let $[A]$ denote the Brauer class of a given Azumaya algebra $A$.
\par

By a result of Schack the Brauer group can alternatively be defined as the collection of Azumaya algebras modulo $R$-linear Morita equivalence~\cite[Theorem 3]{schack92}.  Since $R$-linear Morita equivalences are given by $R$-central invertible bimodules this implies

\begin{proposition}[\cite{schack92}]\label{prop:brauereq}
Given two Azumaya algebras $A$ and $B$ over $R$, we have $[A]=[B]$ in the Brauer group if and only if there exists an invertible $B\ox_RA^{op}$-module $M$.
\end{proposition}

\begin{definition}
We give the Brauer group $\mathrm{Br}(X)$ the $\Aut(X)$-action defined by $\phi\cdot [A]:=[\phi_\ast A]$.
\end{definition}

The following result was first proved by Antieau using techniques from derived algebraic geometry~\cite{antieau}.  Specifically, part (3) is due originally to Antieau.

\begin{corollary}\label{cor:antieau}
Let $A$ and $B$ be Azumaya algebras over $R$, and take $T$ in $\DPic(A,B)$.  Then we have 
\begin{enumerate}
\item $T\cong \Sigma^n M$ for some invertible $B\ox A^{op}$-module $M$ and $n\in \Gamma(X,\underline{\mathbb{Z}})$.
\item The bimodule $M$ provides a Brauer equivalence between $\phi^T_\ast A=\phi^M_\ast A$ and $B$.
\item (Antieau) If $D^b(A)\cong D^b(B)$ and $R$ is projective over $k$ then $B$ is Brauer equivalent to $\phi_\ast A$ for some automorphism $\phi$ of $X$.
\end{enumerate}
\end{corollary}

\begin{proof}
Statement (1) follows from Proposition~\ref{prop:deq_as_meq}, and (2) follows from Lemma~\ref{lem:auts}.  For (3), Rickard tells us that any derived equivalence produces an object $T$ in $\DPic(A,B)$, provided $A$ and $B$ are projective over $k$~\cite{rickard89,rickard91,yekutieli99}.  Proposition~\ref{prop:deq_as_meq} then tells us that $T$ is the shift of an invertible bimodule, so that we may take $\phi=\phi^T$ to get the result.
\end{proof}

\section{The derived Picard group of an Azumaya algebra}
\label{sect:dpicA}

We fix $A$ to be an Azumaya algebra over $R$.  We also fix $X=\Spec(R)$.

\subsection{The derived Picard group}

\begin{lemma}
The assignment
\[
\Phi:\DPic(A)\to \Aut(X),\ \ [T]\mapsto \phi^T
\]
is a group map with image equal to the stabilizer $\Aut(X)_{[A]}$ of the Brauer class of $A$.
\end{lemma}

\begin{proof}
We first show that $\Phi$ is a group map.  By Proposition~\ref{prop:deq_as_meq} it suffices to check that (a) $\phi^M=\phi^{M'}$ when $M\cong M'$ and (b) $\phi^{M\ox N}=\phi^M\phi^N$ for invertible $A$-bimodules $M$, $M'$ and $N$.  Claim (a) follows from the fact that for any bimodule isomorphism $\sigma:M\to M'$ we will have an $R$-bimodule isomorphism
\[
\mathrm{ad}_\sigma:\End_{B\ox A^{op}}(M)\to\End_{B\ox A^{op}}(M').
\]
\par

Let $f_?\in \Aut(R)$ be such that $\phi^?=\Spec(f_?^{-1})$.  Claim (b) follows from the description of $f_M$ as the unique algebra automorphism so that $mr=f_M(r)m$ for each $m\in M$ and $r\in R$.  Indeed, for any monomial $m\ox n\in M\ox N$ we have
\[
(m\ox n)r=m\ox (nr)=(mf_N(r))\ox n=f_Mf_N(r)(m\ox n).
\]
This implies
\[
\phi^{M\ox N}=\Spec((f_Mf_N)^{-1})=\Spec(f_N^{-1}f_M^{-1})=\Spec(f_M^{-1})\Spec(f_N^{-1})=\phi^M\phi^N.
\]
By Proposition~\ref{prop:brauereq} and Corollary~\ref{cor:antieau} the collection of automorphisms of the form $\phi^T$ in $\Aut(X)$ is exactly the collection of automorphisms $\phi$ so that $A$ is Brauer equivalent $\phi_\ast A$, i.e. the stabilizer of the Brauer class of $A$.
\end{proof}

Recall that for any Azumaya algebra $A$ there is a canonical algebra isomorphism
\[
A\ox_R A^{op}\overset{\cong}\to \mathrm{End}_R(A)
\]
given by multiplication on the right and left.  Since $A$ is finitely generated and projective over $R$ we have then the canonical Morita equivalence
\[
A\ox_R-:R\text{-}\mathrm{mod}\to \mathrm{End}_R(A)\text{-}\mathrm{mod}=A\ox_R A^{op}\text{-}\mathrm{mod}.
\]
This Morita equivalence is one of monoidal categories, where we take on the domain the product $\ox_R$ and on the codomain the product $\ox_A$.  Indeed, we have the obvious natural isomorphisms
\[
(A\ox_R M)\ox_A(A\ox_R N)\to A\ox_R(M\ox_R N),\ \ (a\ox m)\ox (b\ox n)\mapsto ab\ox(m\ox n).
\]
(See~\cite[Theorem 2.3.2]{caenepeel98}.)
\par

\begin{definition}
Let $\DPic_R(A)$ denote the group of isoclasses of invertible objects in $\left(D^b(A\ox_R A^{op}),\ox^\L_A\right)$.
\end{definition}

Note that objects in $\DPic_R(A)$ will still be tilting complexes, and hence still shifts of invertible bimodules.  These bimodules will necessarily be $R$-central.

\begin{lemma}\label{lem:314}
There is an exact sequence
\[
1\to \DPic_R(A)\to \DPic(A)\overset{\Phi}\to \Aut(X)_{[A]}\to 1.
\]
\end{lemma}

\begin{proof}
An $R$-central bimodule is the same thing as a $A\ox_R A^{op}$-module.  So the result follows from the fact that $\phi^T=id_X$ if and only if $T$ is (isomorphic to) a complex of $R$-central bimodules, using Proposition~\ref{prop:deq_as_meq}.
\end{proof}

\begin{lemma}
There is a group isomorphism $A\ox_R^\L-:\DPic_R(R)\to \DPic_R(A)$.
\end{lemma}

\begin{proof}
The above discussion tells us that we have a group isomorphism
\[
A\ox_R-:\frac{\{\mathrm{invertible\ }R\text{-}\mathrm{modules}\}}{\cong}\to \frac{\{\mathrm{invertible\ }A\ox_R A^{op}\text{-}\mathrm{modules}\}}{\cong}
\]
and the result with $\DPic$ follows from the fact that any object in $\DPic_R(R)$ or $\DPic_R(A)$ is the shift of an invertible bimodule.
\end{proof}

Note that the global sections functor gives a group isomorphism
\[
\Gamma:\Pic(X)\overset{\cong}\to \frac{\{\mathrm{invertible\ }R\text{-}\mathrm{modules}\}}{\cong}.
\]
By way of this isomorphism we identify $\Pic(X)$ with the collection of invertible $R$-modules, and get a canonical isomorphism $\Gamma(X,\underline{\mathbb{Z}})\times\Pic(X)\to \DPic_R(R)$ from Proposition~\ref{prop:deq_as_meq}.  So the exact sequence of Lemma~\ref{lem:314} can be rewritten as an exact sequence
\[
1\to \Gamma(X,\underline{\mathbb{Z}})\times\Pic(X)\overset{A\ox_R^\L-}\to \DPic(A)\overset{\Phi}\to \Aut(X)_{[A]}\to 1.
\]
We now give our description of $\DPic(A)$.

\begin{theorem}\label{thm:dpicaz}
There is a 2-cocycle $\alpha:\Aut(X)_{[A]}\times\Aut(X)_{[A]}\to \Pic(X)$, dependent on the Brauer class of $A$, and a group isomorphism
\[
\DPic(A)\cong \left(\Gamma(X,\underline{\mathbb{Z}})\times\Pic(X)\right)\rtimes_\alpha \Aut(X)_{[A]}.
\]
\end{theorem}

The semidirect product here is produced with respect to the action of $\Aut(X)$ on $\Gamma(X,\underline{\mathbb{Z}})\times\Pic(X)$ given by pushforward 
\[
\phi\cdot (\Sigma^nL)=\phi_\ast (\Sigma^nL)=\Sigma^{n\phi^{-1}}(\phi_\ast L),
\]
for $\phi\in \Aut(X)$, $L$ in $\Pic(X)$, and $n\in\Gamma(X,\underline{\mathbb{Z}})$.  Note that $\Aut(X)$ does {\it not} act trivially on sections of the constant sheaf.

\begin{proof}
Let $T$ be a shift of an invertible bimodule and take $f_T\in \Aut(R)$ such that $\phi^T=\Spec(f_T^{-1})$.  To ease notation we take $\ox=\ox^\L$.  Then we have for any $L$ in $\Pic(X)$ and $n\in\Gamma(X,\underline{\mathbb{Z}})$
\[
T\ox_A(A\ox_R \Sigma^n L)\cong T\ox_R (\Sigma^n L)\cong (\Sigma^{n'}L')\ox_R T,
\]
where $L'$ is $L$ with the new $R$-action given by $r\cdot l=f_T^{-1}(r)l$ for each $r\in R$ and $l\in L$, and if $\Sigma^nL=\Sigma^{n(\lambda)}e_\lambda L$ then
\[
\Sigma^{n'}L'=\Sigma^{n(\lambda)}f_T(e_\lambda) L'. 
\]
Rather, $L'=\phi_\ast^T L$ and $\Sigma^{n'}=\Sigma^{{n(\phi^T)^{-1}}}$.  Therefore
\[
T\ox_A(A\ox_R \Sigma^nL)\ox_A T^\vee\cong \phi^T_\ast(\Sigma^n L)\ox_R (T\ox_A T^\vee)\cong A\ox_R \phi^T_\ast(\Sigma^n L).
\]
So conjugation corresponds to the pushforward action of 
\[
\Aut(X)_{[A]}\cong \DPic(A)/(\Gamma(X,\underline{\mathbb{Z}})\times\Pic(X))
\]
on $\Gamma(X,\underline{\mathbb{Z}})\times\Pic(X)$.  Since $\Gamma(X,\underline{\mathbb{Z}})\times\Pic(X)$ is abelian, it follows that there is $\Gamma(X,\underline{\mathbb{Z}})\times\Pic(X)$-valued cocycle $\alpha$ with
\[
\DPic(A)=(\Gamma(X,\underline{\mathbb{Z}})\times\Pic(X))\rtimes_\alpha \Aut(X)_{[A]}
\]
(see~\cite[Theorem 6.6.3]{weibel}).  However, since we can always represent an automorphism $\phi\in\Aut(X)_{[A]}$ by an invertible $A\ox A^{op}$-module the cocycle $\alpha$ can in fact be taken to be $\Pic(X)$-valued.
\end{proof}

It is not our opinion that the cocycle $\alpha$ vanishes in general (although it may).  There seems to be a general lack of tractable but nontrivial examples from which we can study this problem.  Whence we ask

\begin{question}
Is there an example of an Azumaya algebra $A$ for which the cocycle $\alpha$ does {\it not} vanish?  If so, what is the content of the associated class $\bar{\alpha}\in H^2(\Aut(X)_{[A]},\Pic(X))$?
\end{question}

One case in which we do know that $\alpha$ vanishes is the case $A=R$.  Here we will have the canonical section of $\Phi$
\[
\Aut(X)=\Aut(X)_{[R]}\to \DPic(R),\ \ \phi\mapsto {_{\bar{\phi}}R},
\]
where $\bar{\phi}$ is the dual ring map to $\phi$ and $_{\bar{\phi}}R$ is the bimodule $R$ with the left action twisted by $\bar{\phi}$.  In this case we get a description of the derived Picard group of a commutative algebra which refines a well established result of Yekutieli, as well as Rouquier-Zimmermann~\cite[Proposition 3.5]{yekutieli99},\cite[Theorem 2.11]{rouquierzimmerman03},\cite[Theorem 6.10]{yekutieli15}.

\begin{corollary}\label{cor:dpiccom}
For any commutative algebra $R$, with $X=\Spec(R)$, we have
\[
\DPic(R)\cong \left(\Gamma(X,\underline{\mathbb{Z}})\times \Pic(X)\right)\rtimes \Aut(X).
\]
\end{corollary}

\subsection{Local rings and CSAs}
\label{sect:CSA}

Let $\mathrm{Out}(A)$ denote the group of outer automorphisms for $A$, i.e. automorphisms modulo inner automorphisms.  In the case in which $R$ is local the restriction map
\[
\Out(A)\to \mathrm{Aut}(R),\ \ f\mapsto f|R
\]
is an inclusion, by Skolem-Noether~\cite[Proposition 1.4]{milne80}.  One can show that, in fact, this embedding gives an canonical isomorphism $\Out(A)\cong \mathrm{Aut}(R)_{[A]}$.  Note also that the Picard group of $X=\Spec(R)$ vanishes.  Also $\Gamma(X,\underline{\mathbb{Z}})=\mathbb{Z}$ since $X$ is connected.  So we get in this case
\[
\DPic(A)\cong \mathbb{Z}\times\Out(A).
\]
This applies in particular to the case of a central simple algebra over a field.

\begin{proposition}\label{prop:numberfields}
Suppose $k$ is a field and $K$ is a finite field extension of $k$.  (For example, take a number field over $\mathbb{Q}$.)  Then for any central simple algebra $A$ over $K$ we have $\DPic(A)\cong \DPic(K)$ if and only if the restriction map $\Out(A)\to \Aut(K)$ is an isomorphism.
\end{proposition}

\begin{proof}
In this case the group of $k$-algebra automorphism $\Aut(K)$ is finite, as is $\Out(A)$.  We can then recover the automorphisms of $K$ as the subgroup of torsion elements in $\DPic(K)$, and similarly find $\Out(A)$ as the torsion elements in $\DPic(A)$.  So, any isomorphism $\DPic(A)\cong \DPic(K)$ induces an isomorphism on torsion elements $\Out(A)\cong \Aut(K)$.  In particular, the two groups have the same order.  It follows that the restriction map $\Out(A)\to\Aut(K)$ is an isomorphism simply because it is an injection.  Conversely, if the restriction map is an isomorphism then $\DPic(A)$ is isomorphic to $\DPic(K)$ since $\DPic(A)=\mathbb{Z}\times \Out(A)$ and $\DPic(K)=\mathbb{Z}\times\Aut(K)$.
\end{proof}

The situation presented in the proposition gives us an idea of where to look for simple examples for which $\DPic(A)$ and $\DPic(R)$ are not isomorphic.  Consider the number field $K=\mathbb{Q}(\sqrt{2})$ and quaternion algebra 
\[
A=(-1,-\sqrt{2})=K\langle i,j\rangle/(i^2=-1,j^2=-\sqrt{2},ij=-ji),
\]
which is a central simple division algebra over $K$.  For $k=\mathbb{Q}$, one can check $\Aut(K)=\mathbb{Z}/2\mathbb{Z}$, $\Out(A)=\{1\}$, and 
\[
\DPic(A)=\mathbb{Z}\ \ \mathrm{while}\ \ \DPic(K)=\mathbb{Z}\times(\mathbb{Z}/2).
\]

\section{An analysis of Weyl algebras in finite characteristic}
\label{sect:Weyl}

Let $k$ be a field of finite characteristic $p$, and consider the Weyl algebra
\[
A_n(k)=k\langle x_1,\dots, x_{2n}\rangle/([x_i,x_j]-\delta_{i,j+n}+\delta_{i+n,j})_{i,j}.
\]
The Weyl algebra can be understood alternatively as the algebra of crystalline differential operators $\mathcal{D}_{\mathbb{A}^n_k}$ on the affine space $\mathbb{A}^n_k=\Spec(k[x_1,\dots,x_n])$, i.e. the universal enveloping algebra of the Lie algebroid of tangent vectors $T_{\mathbb{A}^n_k}$~\cite{rinehart63} (see also~\cite{bmr08}).  Under this identification each generator $x_{n+j}$, for $1\leq j\leq n$, is identified with the global vector field $\frac{\partial}{\partial x_j}\in \mathcal{D}_{\mathbb{A}^n_k}$.
\par

The Weyl algebra in characteristic $p$ is Azumaya over the commutative subalgebra
\[
Z_n(k):=\mathrm{Center}(A_n(k))=k[x_1^p,\dots, x_{2n}^p],
\]
which is in fact a polynomial ring~\cite{revoy73}.  We fix $\mathbb{A}^{2n}_k=\Spec(Z_n(k))$.

\begin{lemma}\label{lem:An0}
The classes $[A_n(k)]$ and $[Z_n(k)]=1_{\mathrm{Br}}$ are distinct in the Brauer group of $\mathbb{A}^{2n}_k$.  Furthermore, for any automorphism $\phi\in\Aut(\mathbb{A}^{2n}_k)$ we have $[\phi_\ast A_n(k)]\neq [Z_n(k)]$.
\end{lemma}

\begin{proof}
Suppose contrarily that $[A_n(k)]=[Z_n(k)]$.  Then by Morita theory and Proposition~\ref{prop:brauereq} there will be some vector bundle (finitely generated projective module) $V$ over $Z_n(k)$ and an algebra isomorphism $f:A_n(k)\to \End_{Z_n(k)}(V)$.  However, after localizing if necessary, the algebra $\End_{Z_n(k)}(V)$ will have nilpotent elements while $A_n(k)$ is well known to be a domain.  So this can not happen, and $[A_n(k)]\neq [Z_n(k)]$.  Similarly each $[\phi_\ast A_n(k)]\neq [Z_n(k)]$.
\end{proof}

For the Weyl algebra we have $\DPic(A_n(k))=\mathbb{Z}\times \Aut(\mathbb{A}^{2n}_k)_{[A_n(k)]}$ and a canonical embedding
\[
\DPic(A_n(k))\to \DPic(Z_n(k)),
\]
since $\mathbb{A}^{2n}_k$ is connected with vanishing Picard group.  

\subsection{Statements of the results}

We will provide the following result.

\begin{theorem}\label{thm:az}
When the characteristic of $k$ is not $2$ the embedding $\DPic(A_n(k))\to \DPic(Z_n(k))$ is not an isomorphism.
\end{theorem}

This is equivalent to the claim that the inclusion $\Aut(\mathbb{A}^{2n}_k)_{[A_n(k)]}\to \Aut(\mathbb{A}^{2n}_k)$ is not an equality.  Thus we focus on the stabilizer of the Brauer class.  We will make use of the following class of automorphisms.

\begin{definition}\label{def:az}
For any $c\in k^\times$, let $\bar{\omega}(c):Z_n(k)\to Z_n(k)$ be the automorphism given on the generators by
\[
\bar{\omega}(c)(x^p_i)=c^{-\omega_i}x^p_i\ \ \mathrm{where}\ \ \omega_i=\left\{\begin{array}{ll}
0 & \mathrm{if}\ i\leq n\\ 1 &\mathrm{if}\ i>n.
\end{array}\right.
\]
Let $\omega(c)$ denote the dual automorphism $\Spec(\bar{\omega}(c))$.
\end{definition}

In words, $\bar{\omega}(c)$ simply scales the last $n$ generators of $Z_n(k)$ by $c^{-1}$.  In the statement of the following proposition we take formally $\omega(0)_\ast A_n(k)=M_{p^n}(Z_n(k))$.  What we prove specifically is

\begin{proposition}\label{prop:az}
When $\mathrm{char}(k)\neq 2$ there is a group embedding
\[
\omega_\ast:k\to \mathrm{Br}(\mathbb{A}^{2n}),\ \ c\mapsto [\omega(c)_\ast A_n(k)],
\]
from the additive group of $k$, and corresponding embedding into the right coset space
\[
k^\times\to \Aut(\mathbb{A}^{2n}_k)/\Aut(\mathbb{A}^{2n}_k)_{[A_n(k)]},\ \  c\mapsto \omega(c)\cdot \Aut(\mathbb{A}^{2n}_k)_{[A_n(k)]}.
\]
\end{proposition}

Since $\omega_\ast$ is an embedding we find in particular that, when $c\neq 1$, we have $[\omega(c)_\ast A_n(k)]\neq [A_n(k)]$.  So none of the $\omega(c)$ stabilize the Brauer class of $A_n(k)$, except $\omega(1)=id_{Z_n(k)}$.  Note that the order of $\omega(c)$ ranges with $k$ and $c$, and can be made arbitrarily large (even infinite) with large $k$.
\par

From Proposition~\ref{prop:az} we will have the obvious

\begin{proof}[Proof of Theorem~\ref{thm:az}]
Assuming the conclusion of Proposition~\ref{prop:az}, the embedding $\DPic(A_n(k))\to \DPic(Z_n(k))$ will not be surjective since none of the elements
\[
(m,\omega(c))\in \mathbb{Z}\times \Aut(\mathbb{A}^{2n}_k)=\DPic(Z_n(k))
\]
will be in its image.
\end{proof}

We will also get from Proposition~\ref{prop:az} the following corollary.

\begin{corollary}\label{cor:az}
When $k$ is an infinite field the stabilizer $\Aut(\mathbb{A}^{2n}_k)_{[A_n(k)]}$ is of infinite index in $\Aut(\mathbb{A}^{2n}_k)$.
\end{corollary}

\begin{proof}
The index is the cardinality of the quotient $\Aut(\mathbb{A}^{2n}_k)/\Aut(\mathbb{A}^{2n}_k)_{[A_n(k)]}$.  So the result follows from the final claim of Proposition~\ref{prop:az}.
\end{proof}

Although we've only considered a relatively small class of automorphisms here, we expect that there are in fact many automorphisms which do not stabilize the Brauer class of $A_n(k)$.  We expect, for example, that Corollary~\ref{cor:az} will hold even when $k$ is not infinite.  We are also curious about the algebraic nature of the stabilizer.
\begin{question}
If we identify $\mathrm{GL}_{2n}(k)$ with the subgroup of linear automorphisms on $\mathbb{A}^{2n}_k$, is the intersection $\Aut(\mathbb{A}^{2n}_k)_{[A_n(k)]}\cap \mathrm{GL}_{2n}(k)$ a locally closed subset in $\mathrm{GL}_{2n}(k)$?
\end{question}

To prove Proposition~\ref{prop:az} we will need a number of lemmas.  It will also be helpful to have a bit more information about the Brauer group.

\subsection{The Brauer group is a group}

Take $X=\Spec(R)$.  The Brauer group $\mathrm{Br}(X)$ of $X$ becomes a group under the tensor product of Azumaya algebras, $[A]\cdot[B]:=[A\ox_R B]$.  The inverse is given by the opposite algebra.  Indeed, we have by the Azumaya property
\[
[A]\cdot [A^{op}]=[A\ox_R A^{op}]=[\End_R(A)]=[R].
\]
One can see~\cite{milne80}, or any number of other references, for more details.

\subsection{Left action by the Brauer class of $A_n(k)$}

We fix $\mathrm{char}(k)=p\neq 2$.  The following lemma will be indicative of a general occurrence.  We present the specific case first since we believe the details are easier to grasp.  One should recall the definition of the automorphisms $\omega(c)$ from Definition~\ref{def:az}.

\begin{lemma}\label{lem:An2}
There is a $Z_n(k)$-algebra isomorphism
\[
A_n(k)\ox_{Z_n(k)}A_n(k)\cong (\omega(2)_\ast A_n(k))\ox_{Z_n(k)}M_{p^n}(Z_n(k)).
\]
\end{lemma}

\begin{proof}
Let $x_i$ and $y_i$ denote the generators of the two copies of $A_n(k)$ in the product $A_n(k)\ox_{Z_n(k)}A_n(k)$, and take $z_i=x_i^p=y_i^p$.  Then $A_n(k)\ox_{Z_n(k)}A_n(k)$ is generated by the elements $\{x_i\}_{i=1}^{2n}\cup\{y_i\}_{i=1}^{2n}$, or alternatively by the elements
\[
\mathbb{B}=\left\{\frac{1}{2^{\epsilon_i}}(x_i+y_i),\frac{1}{2^{\epsilon_i}}(x_i-y_i)\right\}_{i=1}^{n},
\]
where $\epsilon_i=1$ if $i\leq n$ and $0$ if $i>n$.
We take
\[
\zeta_i=\frac{1}{2^{\epsilon_i}}(x_i+y_i)\ \ \mathrm{and}\ \ \alpha_i=\frac{1}{2^{\epsilon_i}}(x_i-y_i)
\]
so that $\mathbb{B}=\{\zeta_i,\alpha_i\}_i$.
\par

We have, for these alternate generators, the relations
\begin{enumerate}
\item[(a)] $[\zeta_i,\zeta_j]=\frac{1}{2}([x_i,x_j]+[y_i,y_j])=\delta_{i,j+n}-\delta_{i+n,j}$
\item[(b)] ${[\zeta_i,\alpha_j]}=\frac{1}{2}([x_i,x_j]-[y_i,y_j])=0$
\item[(c)] ${[\alpha_i,\alpha_j]}=\frac{1}{2}([x_i,x_j]+[y_i,y_j])=\delta_{i,j+n}-\delta_{i+n,j}$
\item[(d)] $\zeta_i^p=\frac{1}{2^{p\epsilon_i}}(x_i^p+y_i^p)=\frac{2}{2^{p\epsilon_i}}z_i=2^{\omega_i}z_i$
\item[(e)] $\alpha_i^p=\frac{1}{2^{p\epsilon_i}}(x_i^p-y_i^p)=0$.
\end{enumerate}
Here $\omega_i=0$ if $i\leq n$ and $1$ if $i>n$, and for the final two relations we use the fact that $x_i$ and $y_i$ commute so that we may use the binomial expansion of the $p$th power.  We've also used the fact that $2^p=2\in \mathbb{F}_p$, since $|\mathbb{F}_p^\times|=(p-1)$, to get relation (d).
\par

As is shown in \cite[Lemma 2]{belovkontsevich07}, the relation (c) and (e) imply that for each $i\leq n$ the subalgebra
\[
k\langle \alpha_i,\alpha_{i+n}\rangle\subset A_n(k)\ox_{Z_n(k)}A_n(k)
\]
is isomorphic to the ring $M_{p}(k)$.  (These subalgebras act on $k[t]/(t^p)$ by the obvious differential operators and the corresponding faithful representation $k\langle \alpha_i,\alpha_{i+n}\rangle\to \End_k(k[t]/(t^p))$ is an algebra isomorphism.)  Since we have a surjective algebra map
\[
k\langle\alpha_1,\alpha_{n+1}\rangle\ox\dots\ox k\langle\alpha_n,\alpha_{2n}\rangle\to k\langle \alpha_1,\dots, \alpha_{2n}\rangle,
\]
by relation (c), we then have
\[
M_{p^n}(k)\cong M_p(k)^{\ox n}\cong k\langle \alpha_1,\dots, \alpha_{2n}\rangle\subset A_n(k)\ox_{Z_n(k)}A_n(k)
\]
via simplicity of the matrix ring.  Changing base to $Z_n(k)$ on the domain then gives a surjection of $Z_n(k)$-algebras $M_{p^n}(Z_n(k))\to Z_n(k)\langle \alpha_1,\dots, \alpha_{2n}\rangle$.  
\par

As for the subalgebra generated by the $\zeta_i$, via the relations (a) and (d) we see that we also get a surjective algebra map $A_n(k)\to k\langle \zeta_1,\dots,\zeta_{2n}\rangle$.  This map becomes $Z_n(k)$-linear after we pushforward by the automorphism $\omega(2)\in\Aut(\mathbb{A}^{2n}_k)$ on the domain.
\par

By way of the two maps constructed above, and the relations (b), we get a surjective $Z_n(k)$-algebra map
\begin{equation}\label{eq:685}
(\omega(2)_\ast A_n(k))\ox_{Z_n(k)}M_{p^n}(Z_n(k))\to Z_n\langle \zeta_i,\alpha_i\rangle_i=A_n(k)\ox_{Z_n(k)}A_n(k).
\end{equation}
Note that both of these algebras are vector bundles of rank $p^{4n}$ over $Z_n(k)$.  Since any surjective $Z_n(k)$-linear map of vector bundles of the same rank will be an isomorphism, the map (\ref{eq:685}) is in fact an isomorphism.
\end{proof}

\begin{lemma}\label{lem:An3}
For any $c,c'\in k$ we have equalities
\begin{equation}\label{eq:701}
[\omega(c)_\ast A_n(k)]\cdot [\omega(c')_\ast A_n(k)]= [\omega(c+c')_\ast A_n(k)]
\end{equation}
in the Brauer group, where we formally take $\omega(0)_\ast A_n(k)=M_{p^n}(Z_n(k))$.  Furthermore, the class $[A_n(k)]$ is order $p$ in the Brauer group.
\end{lemma}

\begin{proof}
Suppose first that $c+c'\neq 0\in k$.  We let $\{z_i\}_{i=1}^{2n}$ be the standard generators for the algebra of functions on $\mathbb{A}^{2n}_k$.  Each generator $z_i$ is now identified with $c^{-\omega_i}x_i^p$ in $\omega(c)_\ast A_n(k)$, as opposed to $x_i^p$.
\par

One finds, as in Lemma~\ref{lem:An2}, that we have an isomorphism
\[
(\omega(c)_\ast A_n(k))\ox_{Z_n(k)}(\omega(c')_\ast A_n(k))\cong(\omega(c+c')_\ast A_n(k))\ox_{Z_n(k)}M_{p^n}(Z_n(k)),
\]
which will imply (\ref{eq:701}).  Specifically, we use the alternate generators
\[
\zeta_i=\frac{1}{(c+c')^{\epsilon_i}}(c^{\epsilon_i}x_i+ {c'}^{\epsilon_i}y_i)\ \ \mathrm{and}\ \ \alpha_i=\frac{1}{(c+c')^{\epsilon_i}}({c'}^{\omega_i} x_i-c^{\omega_i}y_i)
\]
and follow exactly the proof of Lemma~\ref{lem:An2}.  The most important relation to check here is
\[
\zeta_i^p=\frac{c+c'}{(c+c')^{\epsilon_i}}z_i=(c+c')^{\omega_i}z_i.
\]
When $c+c'=0\Leftrightarrow c'/c=-1$ we note that $\omega(-1)_\ast A_n(k)\cong A_n(k)^{op}$ over $Z_n(k)$ via the map $x_i\mapsto (-1)^{\omega_i}x_i$, where $\omega_i$ is as in Definition~\ref{def:az}.  So 
\[
\begin{array}{rl}
[\omega(c)_\ast A_n(k)]\cdot [\omega(c')_\ast A_n(k)]&=[\omega(c)_\ast A_n(k)\ox_{Z_n(k)}\omega(c')_\ast A_n(k)]\\
&=[\omega(c)_\ast \left(A_n(k)\ox_{Z_n(k)}\omega(c'/c)_\ast A_n(k)\right)]\\
&=[\omega(c)_\ast \left(A_n(k)\ox_{Z_n(k)}\omega(-1)_\ast A_n(k)\right)]\\
&=[\omega(c)_\ast M_{p^{2n}}(Z_n(k))]\\
&=[Z_n(k)],
\end{array}
\]
as proposed.
\par

As for the order of $A_n(k)$, we have by the above argument
\[
[A_n(k)]^l=[\omega(\bar{l})_\ast A_n(k)]
\]
for each integer $l$, where $\bar{l}$ denotes its class in $\mathbb{F}_p$.  Since, by Lemma~\ref{lem:An0} $[\omega(c)_\ast A_n(k)]\neq [Z_n(k)]$ for each $c\in\mathbb{F}_p^\times$, we conclude $\mathrm{ord}(A_n(k))=p$.
\end{proof}

\subsection{The proof of Propositions~\ref{prop:az}}

\begin{proof}[Proof of Proposition~\ref{prop:az}]
Lemma~\ref{lem:An3} already tells us that we have a group map 
\begin{equation}\label{eq:745}
\omega_\ast:k\to \mathrm{Br}(\mathbb{A}_k^{2n}),\ \ c\mapsto [\omega(c)_\ast A_n(k)],
\end{equation}
from the additive group for $k$, as claimed.  The fact that $\omega_\ast$ is an embedding follows from the fact that the class of $\omega(c)_\ast A_n(k)$ is nontrivial whenever $c\neq 1$, by Lemma~\ref{lem:An0}.  Since we have the canonical bijection of sets
\[
\begin{array}{c}
\Aut(\mathbb{A}^{2n}_k)/\Aut(\mathbb{A}^{2n}_k)_{[A_n(k)]}\to \Aut(\mathbb{A}^{2n}_k)\cdot [A_n(k)]\\
\phi\cdot \Aut(\mathbb{A}^{2n}_k)_{[A_n(k)]}\mapsto \phi\cdot [A_n(k)]
\end{array}
\]
it follows that the map
\[
k^\times\to \Aut(\mathbb{A}^{2n}_k)/\Aut(\mathbb{A}^{2n}_k)_{[A_n(k)]},\ \ c\mapsto \omega(c)\cdot \Aut(\mathbb{A}^{2n}_k)_{[A_n(k)]}
\]
corresponding to the embedding (\ref{eq:745}) is also an embedding.
\end{proof}

\section*{Acknowledgements}

Thanks to Amnon Yekutieli for pointing out a number of helpful references.  Thanks also to Ben Antieau for a number of thoughtful comments and to James Zhang, whose inquiries are responsible for the materials of Section~\ref{sect:Weyl}.  Thanks to the referee for many helpful comments.

\bibliographystyle{abbrv}

\def\cprime{$'$}

\end{document}